\documentclass[a4j,10pt]{article}
\usepackage{amsmath}
\usepackage{amssymb}
\usepackage{amsthm}
\usepackage{graphicx}
\usepackage{amscd}
\usepackage{enumitem}
\usepackage{mathtools}
\usepackage[all]{xy}
\usepackage{comment}
\numberwithin{equation}{section}
\theoremstyle{definition}
\newtheorem{example}{Example}[section]
\newtheorem{definition}[example]{Definition}
\theoremstyle{plain}

\newtheorem{cor}[example]{Corollary}
\newtheorem{lem}[example]{Lemma}
\newtheorem{theorem}[example]{Theorem}
\newtheorem*{Operation A}{Operation A}
\newtheorem*{Operation B}{Operation B}
\newtheorem*{Operation C}{Operation C}
\newtheorem*{Operation 1}{Operation 1}
\newtheorem*{Operation 2}{Operation 2}
\newtheorem*{Operation 3}{Operation 3}
\newtheorem*{theorem0}{Theorem}

\newtheorem{mainresult}{Main result}
\theoremstyle{remark}
\newtheorem{remark}[example]{Remark}
\DeclareMathOperator{\im}{im}

\DeclareMathOperator{\Gal}{Gal}

\DeclareMathOperator{\Aut}{Aut}

\DeclareMathOperator{\id}{id}
\DeclareMathOperator{\tor}{tor}
\DeclareMathOperator{\rank}{rank}

\DeclareMathOperator{\Z}{\mathbb{Z}}
\begin{document}
\title{On Iwasawa's class number formula for $\Z_p\rtimes\Z_p$-extensions}
\author{Sohei Tateno}
\date{June 26, 2019}
\maketitle

\begin{abstract}
Let $p$ be a prime number. In this paper, we estimate the variation of the sizes of quotients of certain finitely generated $p$-torsion Iwasawa modules, which are closely related to class numbers. We also construct some $\mathbb{Z}_p\rtimes\mathbb{Z}_p$-extensions whose Iwasawa $\mu$-invariant is nonzero. At the end of this paper, we calculate the determinants of some matrices that are related to the groups $\mathbb{Z}_p\rtimes\mathbb{Z}_p$.
\end{abstract}

\section{Introduction}
In \cite{Iwasawa}, Iwasawa proved the following result, which is so called Iwasawa's class number formula.
\begin{theorem0}[{\cite[p.~212]{Iwasawa}}, Theorem 4.]
Let $K_\infty /K$ be a $\mathbb{Z}_p$-extension and $K_n$ be the subfields corresponding to the subgroups $p^n\mathbb{Z}_p$ of $\mathbb{Z}_p$. If we denote $e_n$ for the $p$-exponent of the class number $h(K_n)$ of $K_n$, then there exist some $\mu,\lambda\geq0$ and $\nu\in\mathbb{Z}$ such that
\begin{equation*}
e_n=\mu p^n+\lambda n+\nu
\end{equation*}
for sufficiently large $n$.
\end{theorem0}
This result is known to be the first assymptotic formula which explains the regularity of the variation of the class numbers of certain towers of number fields. In \cite{CM}, Cuoco and Monsky generalized this result to general $\mathbb{Z}_p^d$-extensions.
\begin{theorem0}[{\cite[p.~235]{CM}}, Theorem I.]
Let $K_\infty /K$ be a $\mathbb{Z}_p^d$-extension and $K_n$ be the subfields corresponding to the subgroups $p^n\mathbb{Z}_p^d$ of $\mathbb{Z}_p^d$. If we denote $e_n$ for the $p$-exponent of the class number $h(K_n)$ of $K_n$, then there exist some $\mu,\lambda\geq0$ such that
\begin{equation*}
e_n=\mu p^{dn}+\lambda np^{(d-1)n}+O(p^{(d-1)n}).
\end{equation*}
\end{theorem0}
In \cite{Perbet}, Perbet obtained some similar results on certain $p$-adic Lie groups by using a result of Harris(\cite{Harris}).
\begin{theorem0}[{\cite[p.~837]{Perbet}}, Corollary 3.4.]
Let $G$ be a finitely generated $p$-valued pro-$p$ group with dimension $d$. Let $K_\infty/K$ be a $G$-extension in which only finitely many primes ramify. Let $K_n$ be the subfields corresponding to $G_n:=\{x\in G|\omega(x)>(p-1)^{-1}+n-1\}$, where $\omega$ is the valuation of $G$. If we denote $Cl_n(p)$ for the $p$-Sylow subgroup of the ideal class group of $K_n$, then there exist some $\rho, \mu\geq0$ such that
\begin{equation*}
\#(Cl_n(p)/p^nCl_n(p))=p^{(\rho n+\mu)(G:G_n)+O(np^{n(d-1)})}.
\end{equation*}
\end{theorem0}
By restricting the $p$-adic Lie groups dealt with, Lei succeeded to obtain several more precise results in \cite{Lei}. In particular, the following result is a generalization of the result of Cuoco and Monsky when $d=2$. Let $p$ be an odd prime number and $K$ be a number field which admits a unique prime ideal $\mathfrak{p}$ lying above $p$. Let $K_\infty/K$ be a Galois extension in which $\mathfrak{p}$ is totally ramified and only finitely many primes ramify. We assume that the Galois group $G:=G(K_\infty/K)$ can be written in the form $G=H\rtimes\Gamma$, where $H$ and $\Gamma$ are subgroups of $G$ isomorphic to the additive group $\mathbb{Z}_p$. Then there exists a $H$-subextension $K_\infty/K^c$ of $K_\infty/K$ such that $G(K^c/K)\simeq\Gamma$. Assume that every prime of $K$ which ramifies in $K_\infty/K$ splits into finitely many primes in $K^c$. Let $L_\infty$ be the maximal unramified abelian pro-$p$ extension of $K_\infty$ and put $X:=G(L_\infty/K_\infty)$. Then $X$ becomes a finitely generated $\Lambda(G)$-module, where $\Lambda(G):=\mathbb{Z}_p[\![G]\!]$. We put $X(p):=\tor_{\mathbb{Z}_p}X$, $\Gamma_m:=\Gamma^{p^m}$, $H_n:=H^{p^n}$, and $\lambda_G:=\rank_{\Lambda(H)}(X/X(p))$. Let $K_n$ be the subfields of $K_\infty/K$ corresponding to the subgroups $H_n\rtimes\Gamma_n$ of $G$ and $e_n$ denote the $p$-exponent of the class number of $K_n$.

\begin{theorem0}[{\cite[p.~361]{Lei}}, Corollary 5.3.]
Assume that $X$ is finitely generated over $\Lambda(H)$. Then one has
\begin{equation*}
e_n=\lambda_G np^n+O(p^n).
\end{equation*}
\end{theorem0}

The assumption $X$ being finitely generated over $\Lambda(H)$ implies that the so-called Iwasawa $\mu_G$-invariant is equal to zero. The purpose of this paper is to consider some estimates of the variation of the sizes of quotients of certain finitely generated $p$-torsion $\Lambda(G)$-modules, which are closely related to class numbers, when $\mu_G$ is not necessarily zero.

Let $X$ be a finitely generated $p$-torsion $\Lambda(G)$-module and $I_{\Gamma_n}, I_{H_m}$ the kernel of the natural surjective homomorphisms of rings $\mathbb{Z}_p[\![\Gamma_n]\!]\twoheadrightarrow\mathbb{Z}_p, \mathbb{Z}_p[\![H_m]\!]\twoheadrightarrow\mathbb{Z}_p$ respectively. Then, by Lemma \ref{newproof1}, $X_{H_m}:=X/I_{H_m}X$ becomes a finitely generated $\Lambda(\Gamma)$-module for each $m\geq0$. By \cite{Venjakob2}[p.~295, Theorem 3.40.], one has an exact sequence of finitely generated $\Lambda(G)$-modules
\begin{equation*}\label{neweq}
0\rightarrow A\rightarrow X\buildrel{\varphi}\over\rightarrow\bigoplus_{i=1}^s\Lambda(G)/p^{m_i}\Lambda(G)\rightarrow  B\rightarrow0
\end{equation*}
with $A$,$B$ pseudo-null. We also assume that $A$,$B$ are finitely generated over $\Lambda(H)$. Then we have

\begin{mainresult}[Theorem \ref{doctortheorem}]
Let $X'_{H_m}$ be the maximal finite $\Lambda(\Gamma)$-submodule of $X_{H_m}$. Then there exist $\mu_A,\nu_A$ depending on $A$ and $\nu_B\in\mathbb{Z}$ depending on $B$ such that
\begin{equation*}
\# X'_{H_m}=p^{\mu_Ap^m+\nu_A+\nu_B}.
\end{equation*}
for sufficiently large $m$.
\end{mainresult}

\begin{mainresult}[Corollary \ref{corcor}]
We have
\begin{equation*}
p^{\mu_Gp^{2n}}\leq\#(X_{H_n})_{\Gamma_n}\leq p^{\mu_Gp^{2n}+\mu_Ap^n+\nu_A+\nu_B}
\end{equation*}
for all $n\geq0$.
\end{mainresult}

Although there are some $\Lambda(G)$-modules satisfying the assumptions of these main results with $\mu_G\neq0$ in the ring theoretical settings, actual $\mathbb{Z}_p\rtimes\mathbb{Z}_p$-extensions whose Iwasawa module satisfies the assumptions with $\mu_G\neq0$ have not been found as far as the author knows. To begin with dealing with this problem, we will also construct some $\mathbb{Z}_p\rtimes\mathbb{Z}_p$-extensions with $\mu_G\neq0$. We have the following main result, which is a partial analogue of Iwasawa's result (\cite{Iwasawa2}[p.~6, Theorem 1.]) for $\mathbb{Z}_p$-extensions.

\begin{mainresult}[Theorem \ref{0118}]
Let $p=3$ and $K:=\mathbb{Q}(\zeta_3)$. For any $N\geq1$, there exists a cyclic extension $K'/K$ with degree $3$ and $L'/K'$ with $\Gal(L'/K')\simeq\mathbb{Z}_p\rtimes\mathbb{Z}_p$ such that $\mu_G(L'/K')\geq N$.
\end{mainresult}

At the end of this paper, we would like to calculate the determinants of some matrices that are related to the groups $\mathbb{Z}_p\rtimes\mathbb{Z}_p$. Let $p$ be a prime number and $n,d\geq 1$ integers. Let $u$ be an integer such that $(p,u)=1$. Let $A(p,n,d,u)=(a_{ij})\in M_{p^n}(\mathbb{Z}_p)$ defined by
\begin{equation*}
a_{ij}=
\begin{cases}
1&j\equiv(i-1)(1+pu)^k+1\mod p^n(0\leq k\leq d-1)\\
0&\text{otherwise}.
\end{cases}
\end{equation*}

\begin{mainresult}[Theorem \ref{03122}]
We have
\begin{equation*}
|A|=
\begin{cases}
1&n=1\\
d^{(p-1)(n-1)}&n\neq1, d<p\\
0&n\neq1, d\geq p.
\end{cases}
\end{equation*}
\end{mainresult}

\section*{Acknowledgement}

I would like to thank Hiroshi Suzuki most warmly for his steady guidance and helpful comments as my advisor. I am immensely indebted to Takenori Kataoka for helping me considering Theorem \ref{0118} together. I am also very grateful to Antonio Lei and Otmar Venjakob for not only answering my numerous questions throughout Iwasawa 2017 but also answering my additional questions even after Iwasawa 2017 via e-mail. Also, I cannot thank Takashi Hara too much for teaching me this field and giving me a lot of advice. Yasushi Mizusawa is also thanked for answering my questions carefully and clearly. Tomohiro Ikkai, Haibo Jin, Yuta Suzuki, Oliver Thomas, and Kota Yamamoto are thanked for very informatic discussions during the preparation of this paper. Finally, I greatly appreciate the continued support of my parents.

\section{Some properties of $\Lambda(G)$-modules}

In this section, we shall study some properties of $\Lambda(G)$-modules which will be used in the proof of our main results.

 Let $G$ be a profinite group such that $G=H\rtimes_\phi\Gamma$, where $H$ and $\Gamma$ are multiplicative groups which are isomorphic to the additive group $\mathbb{Z}_p$ and $\phi:\Gamma\rightarrow \Aut (H)$ is a fixed group homomorphism. We fix topological generators $h$ of $H$ and $\gamma$ of $\Gamma$. Define the ring automorphism $\sigma:\mathbb{Z}_p[[S]]\rightarrow\mathbb{Z}_p[[S]]$ by $\sigma(S):=\sum_{i=1}^\infty\binom{\phi(\gamma)}{i}S^i$ and the group homomorphism $\delta:\mathbb{Z}_p[[S]]\rightarrow\mathbb{Z}_p[[S]]$ by $\delta:=\sigma-\id$. Then, by \cite{Venjakob1}[p.~157, Example 2.2.], we obtain a ring isomorphism
\begin{equation*}
\Lambda=\Lambda(G):=\mathbb{Z}_p[\![G]\!]\simeq\mathbb{Z}_p[[S]][[T;\sigma,\delta]]=:\mathbb{Z}_p[[S,T;\sigma,\delta]]
\end{equation*}
where $h$ and $\gamma$ are corresponding to $1+S$ and $1+T$ respectively. Recall that the multiplication of $\mathbb{Z}_p[[S,T;\sigma,\delta]]$ is given by
\begin{equation*}
TF(S)=\sigma(F(S))T+\delta(F(S))
\end{equation*}
for $F(S)\in\mathbb{Z}_p[[S]]$. The topology of $\Lambda(G)$ is defined in section 7.1 of \cite{Dixon}. %One can check that $G$ is a finitely generated $p$-valued profinite group and $p$-adic analytic group of dimension $2$. 
By \cite{Venjakob1}[p.~158, Example 2.3.], $G$ is a uniform pro-$p$ group. Therefore, by \cite{Dixon}[p.~161, 7.25 Corollary], $\Lambda(G)$ is a non-commutative noetherian integral domain. Let $I_{\Gamma_n}, I_{H_m}$ be the kernel of the natural surjective homomorphisms of rings $\mathbb{Z}_p[\![\Gamma_n]\!]\twoheadrightarrow\mathbb{Z}_p, \mathbb{Z}_p[\![H_m]\!]\twoheadrightarrow\mathbb{Z}_p$ respectively. Put $\omega_m:=(1+S)^{p^m}-1\in\mathbb{Z}_p[[S]]\simeq\Lambda(H)$. Then we find $I_{H_m}\Lambda(G)=\omega_m\Lambda(G)$. When we say $X$ is a $\Lambda(G)$-module, we always understand that $X$ carries the structure of a Hausdorff abelian topological group and the structure of a left $\Lambda(G)$-module such that the action $\Lambda(G)\times X\rightarrow X$ is continuous. %Furthermore, by \cite{Goodearl}[p.~112, Corollary 6.7.], it satisfies Ore condition.

\begin{lem}\label{newproof1}
\begin{enumerate}[label=$(\roman*)$]
\item
One has $I_{H_m}\Lambda(G)=\Lambda(G)I_{H_m}$.
\item
$\Lambda(G)/I_{H_m}\Lambda(G)$ is finitely generated over $\Lambda(\Gamma)$ as a left $\Lambda(\Gamma)$-module.
\end{enumerate}
\end{lem}
\begin{proof}
\begin{enumerate}[label=$(\roman*)$]
\item
This is proved in \cite{Venjakob1}[p.~181, Proposition 7.6.].
\item
By \cite{Venjakob1}[p.~158, Example 2.3.], $G$ is a uniform pro-$p$ group. Hence, for any $r\in\Lambda(G)$, \cite{Dixon}[p.~155, 7.20 Theorem] implies that $r$ can be uniquely written in the form
\begin{equation*}
r=\sum_{j=0}^\infty F_i(T)S^i
\end{equation*}
where $F_i(T)\in\mathbb{Z}_p[[T]]\simeq\Lambda(\Gamma)$. By ($i$), we have $I_{H_m}\Lambda(G)=\Lambda(G)I_{H_m}$. Since $\Lambda(G)I_{H_m}=\Lambda(G)\omega_m$, we have the congruence
\begin{equation*}
F_i(T)S^i\equiv F_i(T)(-\sum_{j=1}^{p^m-1}\binom{p^m}{j}S^j)S^{i-p^m}\pmod{\Lambda(G)I_{H_m}}
\end{equation*}
for each $i\geq p^m$. By using this congruence many times, we obtain some $F'_0(T), \cdots, F'_{p^m-1}(T)\in\mathbb{Z}_p[[T]]\simeq\Lambda(\Gamma)$ such that $r=\sum_{i=0}^{p^m-1}F'_i(T)S^i$. Note that each coefficient of $F'_i(T)$ converges in $\mathbb{Z}_p$ because $\mathbb{Z}_p$ is equipped with $p$-adic topology. This completes the proof.
\end{enumerate}
\end{proof}
If $X$ is a finitely generated $\Lambda(G)$-module, then, by Lemma \ref{newproof1}, $X_{H_m}:=X/I_{H_m}X$ becomes a finitely generated $\Lambda(\Gamma)$-module for each $m\geq0$. Hence there exists the maximal finite $\Lambda(\Gamma)$-submodule $X'_{H_m}$ of $X_{H_m}$ for each $m\geq0$.

\begin{lem}\label{newproof2}
For each $m\geq0$, we have
\begin{equation*}
\bigoplus_{i=1}^s(\Lambda(G)/p^{m_i}\Lambda(G))_{H_m}\simeq\bigoplus_{i=1}^s(\Lambda(\Gamma)/p^{m_i}\Lambda(\Gamma))^{p^m}
\end{equation*}
as $\Lambda(\Gamma)$-modules.
\end{lem}
\begin{proof}
We have
\begin{eqnarray*}
&(\Lambda(G)/p^{m_i}\Lambda(G))_{H_m}\\
\simeq&\Lambda(G)/(p^{m_i}\Lambda(G)+\omega_m\Lambda(G))\\
\simeq&(\Lambda(G)/\omega_m\Lambda(G))/p^{m_i}(\Lambda(G)/\omega_m\Lambda(G)).
\end{eqnarray*}
Since we have
\begin{equation*}
\Lambda(G)/\omega_m\Lambda(G)\simeq\Lambda(\Gamma)^{p^m}
\end{equation*}
as a $\Lambda(\Gamma)$-module, we obtain
\begin{eqnarray*}
(\Lambda(G)/p^{m_i}\Lambda(G))_{H_m}&\simeq&\Lambda(\Gamma)^{p^m}/p^{m_i}\Lambda(\Gamma)^{p^m}\\
&\simeq&(\Lambda(\Gamma)/p^{m_i}\Lambda(\Gamma))^{p^m}.
\end{eqnarray*}
as a $\Lambda(\Gamma)$-module. Therefore, one has
\begin{equation*}
\bigoplus_{i=1}^s(\Lambda(G)/p^{m_i}\Lambda(G))_{H_m}\simeq\bigoplus_{i=1}^s(\Lambda(\Gamma)/p^{m_i}\Lambda(\Gamma))^{p^m}
\end{equation*}
\end{proof}

\begin{lem}\label{injection}
\begin{equation*}
\Lambda(G)/p^k\Lambda(G)\buildrel{\omega_m(S)\times}\over\longrightarrow\Lambda(G)/p^k\Lambda(G)
\end{equation*}
is injective.
\end{lem}
\begin{proof}
For any $f(S,T)\in\Lambda(G)$, if $\omega_m(S)f(S,T)\in p^k\Lambda(G)$, then $g(S,T)\in\Lambda(G)$ such that $\omega_m(S)f(S,T)=p^kg(S,T)$. If we put $\xi_m:=\frac{\omega_m}{\omega_{m-1}}(\xi_0=\omega_0)$, then one can write $\omega_m=\xi_m\xi_{m-1}\cdots\xi_{0}$. By \cite{Venjakob1}[p.~182, Proposition 7.6.], $\xi_m$ are prime elements of $\Lambda(G)$, i.e., one has $\xi_m\Lambda=\Lambda\xi_m$ and $\xi_m\Lambda$ are prime ideals. Since $\xi_m\Lambda\supset p^k\Lambda\cdot\Lambda g(S,T)\Lambda$, one has either $\xi_m\Lambda\supset p^k\Lambda$ or $\xi_m\Lambda\supset\Lambda g(S,T)\Lambda$. If $\xi_m\Lambda\supset p^k\Lambda$, then there exists $h(S,T)\in\Lambda(G)$ such that $\xi_mh(S,T)=p^k$. But this is impossible. Hence $\xi_m\Lambda\supset\Lambda g(S,T)\Lambda$, and so there exists $g_1(S,T)\in\Lambda(G)$ such that $\xi_m g_1(S,T)=g(S,T)$. Therefore, we have
\begin{equation*}
\xi_m\cdots\xi_0f(S,T)=\xi_mg_1(S,T)p^k
\end{equation*}
Since $\Lambda(G)$ is integral, one can erase $\xi_m$. By iterating this, one obtains
\begin{equation*}
f(S,T)=g_{m+1}(S,T)p^k
\end{equation*}
for some $g_{m+1}\in\Lambda(G)$.
Hence $f(S,T)\in p^k\Lambda(G)$.
\end{proof}

\begin{lem}\label{0110}
Let $K$ be a number field, $K_\infty/K$ a $\mathbb{Z}_p$-extension, and $L_\infty/K_\infty$ the maximal $p$-ramified abelian pro-$p$ extension. Let $L/K_\infty$ be a Galois subextension of $L_\infty/K_\infty$. Then $\Gal(L/K)$ is abelian if and only if $K_\infty/K$ acts trivially on $L/K_\infty$.
\end{lem}
\begin{proof}
Assume that $K_\infty/K$ acts trivially on $L/K_\infty$. Consider the exact sequence of topological groups
\begin{equation*}
0\rightarrow\Gal(L/K_\infty)\rightarrow \Gal(L/K)\buildrel{j}\over\rightarrow\Gal(K_\infty/K)\rightarrow0.
\end{equation*}
Let $\tilde{\gamma}\in\Gal(L/K)$ be a lifting of the fixed topological generator $\gamma\in\Gamma=\Gal(K_\infty/K)$. Then one can define the homomorphism $\varphi:\Gal(K_\infty/K)\ni\gamma\mapsto\tilde{\gamma}\in\Gal(L/K)$. Since $j\circ\varphi=\id_\Gamma$, we obtain
\begin{equation*}
\Gal(L/K)=\Gal(L/K_\infty)\rtimes\Gal(K_\infty/K).
\end{equation*}
For any $x_1, x_2\in\Gal(L/K_\infty)$, there exist $y_1, y_2\in\Gal(L/K_\infty)$ and $\gamma_1, \gamma_2\in\Gamma$ such that $x_i=(y_i,\gamma_i)$ for $i=1,2$. Hence
\begin{eqnarray*}
x_1\cdot x_2&=&(y_1,\gamma_1)(y_2,\gamma_2)=(y_1\tilde{\gamma}_1y_2\tilde{\gamma}_1^{-1}, \gamma_1\gamma_2)\\&=&(y_1y_2,\gamma_1\gamma_2)=x_2\cdot x_1.
\end{eqnarray*}
Thus $\Gal(L/K)$ is abelian. Conversely, if $\Gal(L/K)$ is abelian, then one has $\tilde{\gamma}'x\tilde{\gamma}'^{-1}=x$ for all $x\in\Gal(L/K_\infty)$ and $\gamma'\in\Gamma$. This completes the proof.
\end{proof}

\begin{definition}
Let $X$ be a finitely generated $\Lambda(G)$-module. Then, 
by \cite{Venjakob2}[p.~295, Theorem 3.40.], we have a pseudo-isomorphism 
\begin{equation*}
X(p)\sim\bigoplus_{i=1}^s\Lambda(G)/p^{m_i}\Lambda(G),
\end{equation*}
where $X(p):=\tor_{\mathbb{Z}_p}X$. $\mu_G:=\sum_{i=1}^s m_i$ is called the \textbf{Iwasawa $\mu_G$-invariant of $X$}.
\end{definition}

\section{The main result}

In this section, we prove one of our main results(Corollary \ref{corcor}) by using propositions and lemmas which we have proved already in Section 2.

Throughout this section, let $X$ be a $p$-torsion finitely generated $\Lambda(G)$-module. By \cite{Venjakob2}[p.~295, Theorem 3.40.], we have an exact sequence of finitely generated $\Lambda(G)$-modules
\begin{equation}\label{neweq}
0\rightarrow A\rightarrow X\buildrel{\varphi}\over\rightarrow\bigoplus_{i=1}^s\Lambda(G)/p^{m_i}\Lambda(G)\rightarrow  B\rightarrow0
\end{equation}
with $A$,$B$ pseudo-null. We assume that $A$,$B$ are finitely generated over $\Lambda(H)$.

\begin{theorem}\label{doctortheorem}
Let $X'_{H_m}$ be the maximal finite $\Lambda(\Gamma)$-submodule of $X_{H_m}$. Then there exist $\mu_A,\nu_A$ depending on $A$ and $\nu_B\in\mathbb{Z}$ depending on $B$ such that
\begin{equation*}
\# X'_{H_m}=p^{\mu_Ap^m+\nu_A+\nu_B}.
\end{equation*}
for sufficiently large $m$.
\end{theorem}
\begin{proof}
By \eqref{neweq}, we have two exact sequences of $\Lambda(G)$-modules
\begin{equation*}
0\rightarrow A\rightarrow X\rightarrow\im\varphi\rightarrow0
\end{equation*}
and
\begin{equation*}
X\rightarrow Y\rightarrow B\rightarrow 0,
\end{equation*}
where we put $Y:=\bigoplus_{i=1}^s\Lambda(G)/p^{m_i}\Lambda(G)$. For any $m\geq0$, consider the exact commutative diagram
\[
\xymatrix{
0\ar[r]&A\ar[r]\ar[d]&X\ar[r]\ar[d]_{\omega_m(S)\times}&\im\varphi\ar[r]\ar[d]&0\\
0\ar[r]&A\ar[r]&X\ar[r]&\im\varphi\ar[r]&0.
}
\]
Since Lemma \ref{injection} ensures the injectivity of $\omega_m(S)\times:\im\varphi\rightarrow\im\varphi$, by the snake lemma, we obtain the exact sequence of $\Lambda(\Gamma)$-modules
\begin{equation}\label{neweq2}
0\rightarrow A_{H_m}\rightarrow X_{H_m}\buildrel{\overline{\varphi}}\over\rightarrow(\im\varphi)_{H_m}\rightarrow0
\end{equation}

Since taking $/I_{H_m}$ is right exact, we also have the exact sequence of $\Lambda(\Gamma)$-modules
\begin{equation}\label{neweq3}
X_{H_m}\buildrel{\varphi_{H_m}}\over\rightarrow Y_{H_m}\rightarrow B_{H_m}\rightarrow0.
\end{equation}
By assumption, $A$ and $B$ are finitely generated torsion $\mathbb{Z}_p$-modules. Therefore, we have $\# A_{H_m}=p^{\mu_Ap^m+\nu_A}$ for large $m$ and $\# B_{H_m}=p^{O(p^m)}$. Consider the commutative diagram of $\Lambda(\Gamma)$-modules
\[
\xymatrix{
X \ar[r]^{\varphi} \ar[d]_{\pi} & Y \ar[d]^{\pi} \\
X_{H_m} \ar[r]^{\varphi_{H_m}} & Y_{H_m}
}
\]
We can define the well-defined surjective homomorphism of $\Lambda(\Gamma)$-modules
\begin{equation*}
\begin{matrix}
\psi:&(\im\varphi)_{H_m}&\longrightarrow&\im(\varphi_{H_m})\\
&\rotatebox{90}{$\in$}&&\rotatebox{90}{$\in$}\\
&\varphi(x)+I_{H_m}\im\varphi&\longmapsto&\varphi_{H_m}(x+I_{H_m}X)=\varphi(x)+I_{H_m}Y.
\end{matrix}
\end{equation*}
with $\ker\psi=(I_{H_m}Y\cap\im\varphi)/I_{H_m}\im\varphi$. Hence we obtain the commutative diagram of $\Lambda(\Gamma)$-modules

\[
\xymatrix{
& & 0 \ar[d] & & \\
& & \ker\overline{\varphi} \ar@{^{(}->}[ld] \ar[d] & & \\
0 \ar[r] & \ker(\varphi_{H_m}) \ar[r] & X_{H_m} \ar[r]^{\varphi_{H_m}} \ar[d]^{\overline{\varphi}} & \im(\varphi_{H_m}) \ar[r] & 0 \\
& & (\im\varphi)_{H_m} \ar@{->>}[ru]_{\psi} \ar[d]  & & \\
& & 0 & &
}
\]
Therefore, we can consider the commutative diagram of $\Lambda(\Gamma)$-modules
\[
\xymatrix{
0\ar[r]&\ker(\varphi_{H_m})/\ker\overline{\varphi} \ar[r]&X_{H_m}/\ker\overline{\varphi} \ar[r]\ar[d]&X_{H_m}/\ker(\varphi_{H_m})\ar[r]\ar[d]&0\\
0\ar[r]&(I_{H_m}Y\cap\im\varphi)/I_{H_m}\im\varphi \ar[r]&(\im\varphi)_{H_m} \ar[r]&\im(\varphi_{H_m})\ar[r]&0
}
\]
This implies that we have
\begin{equation}\label{neweq4}
\ker(\varphi_{H_m})/\ker\overline{\varphi}\simeq (I_{H_m}Y\cap\im\varphi)/I_{H_m}\im\varphi.
\end{equation}

For each $m\geq0$, define the homomorphism of $\Lambda(H)$-modules $\Omega_m$ by
\begin{equation*}
\begin{matrix}
\Omega_m:&Y&\longrightarrow&Y\\
&\rotatebox{90}{$\in$}&&\rotatebox{90}{$\in$}\\
&a&\longmapsto&\omega_m a.
\end{matrix}
\end{equation*}
Then we have
\begin{equation*}
\Omega_m^{-1}(\im\varphi)\subset\Omega_{m+1}^{-1}(\im\varphi)
\end{equation*}
for each $m\geq0$. Indeed, if $a\in\Omega_m^{-1}(\im\varphi)$, then one has $\omega_m a\in\im\varphi$. Since $\frac{\omega_{m+1}}{\omega_m}\in\Lambda(H)$ and $\im\varphi$ is a $\Lambda(H)$-module, we have $\omega_{m+1} a\in\im\varphi$. Thus we obtain $a\in \Omega_{m+1}^{-1}(\im\varphi)$. Therefore, $\Omega_m^{-1}(\im\varphi)/\im\varphi$ are submodules of  the finitely generated $\Lambda(H)$-module $Y/\im\varphi(=B)$ such that we have
\begin{equation*}
\Omega_m^{-1}(\im\varphi)/\im\varphi\subset\Omega_{m+1}^{-1}(\im\varphi)/\im\varphi
\end{equation*}
for each $m\geq0$. Hence we obtain a sequence of finitely generated $\Lambda(H)$-modules
\begin{equation*}
\Omega_0^{-1}(\im\varphi)/\im\varphi\subset\Omega_1^{-1}(\im\varphi)/\im\varphi\subset\cdots.
\end{equation*}
Since $\Lambda(H)$ is noetherian, there exists some $\alpha\geq0$ such that one has
\begin{equation}\label{neweq6}
\Omega_m^{-1}(\im\varphi)/\im\varphi=\Omega_{\alpha}^{-1}(\im\varphi)/\im\varphi
\end{equation}
for all $m\geq\alpha$. Since $I_{H_m}Y=\omega_mY$, we can consider the homomorphism of $\Lambda(H)$-modules
\begin{equation*}
\begin{matrix}
\omega_m\times:&\Omega_m^{-1}(\im\varphi)/\im\varphi&\longrightarrow& (I_{H_m}Y\cap\im\varphi)/I_{H_m}\im\varphi\\
&\rotatebox{90}{$\in$}&&\rotatebox{90}{$\in$}\\
&a+\im\varphi&\longmapsto&\omega_ma+I_{H_m}\im\varphi.
\end{matrix}
\end{equation*}
for each $m\geq0$. This is surjective. Indeed, if $x+I_{H_m}\im\varphi\in(I_{H_m}Y\cap\im\varphi)/I_{H_m}\im\varphi$, then there exist some $a\in Y$ and $b\in\im\varphi$ such that we have $x=\omega_ma=\varphi(b)$. Hence we have $a\in\Omega_m^{-1}(\im\varphi)$. By Lemma \ref{injection}, one can also check that this is injective, and so this is bijective. Since $Y/\im\varphi(=B)$ is finitely generated and torsion over $\Lambda(H)$, so is $\Omega_\alpha^{-1}(\im\varphi)/\im\varphi$. Hence ($\Omega_\alpha^{-1}(\im\varphi)/\im\varphi)_{H_m}$ is finite for all $m\geq\alpha$. Note that we have
\begin{equation*}
\Omega_\alpha^{-1}(\im\varphi)/\im\varphi=({\Omega_\alpha^{-1}(\im\varphi)/\im\varphi})_{H_m}.
\end{equation*}
Therefore, there exists some $\nu_B\geq0$ such that $\#(I_{H_m}Y\cap\im\varphi)/I_{H_m}\im\varphi=p^{\nu_B}$ for all $m\geq\alpha$, that is, $\#(\ker(\varphi_{H_m})/\ker\overline{\varphi})=p^{\nu_B}$ for all $m\geq\alpha$.
Since $\#\ker\overline{\varphi}\buildrel{\eqref{neweq2}}\over=\# A_{H_m}=p^{\mu_Ap^m+\nu_A}$, this implies that
\begin{equation*}
\#\ker(\varphi_{H_m})= p^{\nu_B}p^{\mu_Ap^m+\nu_A}=p^{\mu_Ap^m+\nu_A+\nu_B}.
\end{equation*}
By \eqref{neweq3}, we obtain the exact sequence of $\Lambda(\Gamma)$-modules
\begin{equation*}
0\rightarrow\ker(\varphi_{H_m})\rightarrow X_{H_m}\rightarrow Y_{H_m} \rightarrow B_{H_m}\rightarrow0
\end{equation*}
with $\ker(\varphi_{H_m}),B_{H_m}$ finite. Note that $Y_{H_m}$ are elementary Iwasawa modules. By classical Iwasawa theory, this implies that $\ker(\varphi_{H_m})$ is the maximal finite $\Lambda(\Gamma)$-module of $X_{H_m}$, i.e., we have $X'_{H_m}=\ker(\varphi_{H_m})$. Consequently, we obtain $X'_{H_m}=p^{\mu_Ap^m+\nu_A+\nu_B}$.
\end{proof}

\begin{cor}\label{corcor}
We have
\begin{equation*}
p^{\mu_Gp^{2n}}\leq\#(X_{H_n})_{\Gamma_n}\leq p^{\mu_Gp^{2n}+\mu_Ap^n+\nu_A+\nu_B}
\end{equation*}
for all $n\geq0$.
\end{cor}
\begin{proof}
By Theorem \ref{doctortheorem}, we have the exact sequence of $\Lambda(\Gamma)$-modules
\begin{equation*}
0\rightarrow \ker(\varphi_{H_n})\rightarrow X_{H_n}\rightarrow Y_{H_n} \rightarrow B_{H_n}\rightarrow0
\end{equation*}
for each $n\geq0$. Since $Y_{H_n}$ are elementary Iwasawa modules, by classical Iwasawa theory, we obtain the assertion.
\end{proof}

\begin{remark}

Let $p$ be an odd prime number and $K$ a number field which admits a unique prime ideal $\mathfrak{p}$ lying above $p$. Consider a Galois extension $K_\infty/K$ satisfying $G(K_\infty/K)=G$, $\mathfrak{p}$ is totally ramified in $K_\infty/K$, and $K_\infty/K$ is unramified outside $\mathfrak{p}$. Then there exists a $H$-subextension $K_\infty/K^c$ of $K_\infty/K$ such that $G(K^c/K)\simeq\Gamma$. For each $n\geq0$, we denote $H_n:=H^{p^n}$, $\Gamma_n:=\Gamma^{p^n}$, $G_n:=H_n\rtimes\Gamma_n$, and $K_n:=K_\infty^{G_n}$. Let $L_\infty$ be the maximal unramified abelian pro-$p$ extension of $K_\infty$ and put $X:=G(L_\infty/K_\infty)$. Then, by \cite{Perbet}[p.~835, Proposition 3.1.], $X$ becomes a finitely generated $\Lambda(G)$-module.  For $n\geq0$, let $e_n$ denote the $p$-exponent of the class number $h(K_n)$ of $K_n$. Since $K_\infty/K$ is unramified outside $\mathfrak{p}$, by \cite{Lei}[p.~5, the short exact sequence (2.3) and Corollary 2.6.], we find that $e_n$ is equal to the $p$-exponent of $\#(X_{\Gamma_n})_{H_n}$. Hence, if there exists a $\mathbb{Z}_p\rtimes\mathbb{Z}_p$-extension $K_{\infty}/K$ such that the corresponding Iwasawa module $X$ is a $p$-torsion module satisfying the condition for Theorem \ref{corcor}, then one has
\begin{equation*}
\mu_G p^{2n}\leq e_n\leq \mu_G p^{2n}+\mu_A p^n+\nu_A+\nu_B.
\end{equation*}
for large $n$.

\end{remark}

\section{Construction of $\mathbb{Z}_p\rtimes\mathbb{Z}_p$-extensions with $\mu_G\neq0$}

In this section, we shall construct some $\mathbb{Z}_p\rtimes\mathbb{Z}_p$-extensions with $\mu_G\neq0$.

\begin{definition}
Let $K:=\mathbb{Q}(\sqrt{-d})$ be an imaginary quadratic field. By \cite{Washington}[p.~266, Theorem 13.4], there is a $\mathbb{Z}_p^2$-extension $K_\infty/K$. Put $\Gamma':=G(K_\infty/K)$ and $G:=G(K/\mathbb{Q})$. Then one finds that $G$ acts on $\Gamma'$. $\Gamma'$ can be written in the form $\Gamma'=\Gamma^+\oplus\Gamma^-$, where $\Gamma^+$ is the part $G$ acts trivially. The $\Gamma^-$-extension is called \textbf{the anticyclotomic $\mathbb{Z}_p$-extension of $K$}.
\end{definition}

From now on, let $K$, $K_\infty^-$ be the above fields. Let $M_\infty/K_\infty^-$ be the maximal pro-$p$ abelian $p$-ramified extension. Since the Leopoldt's conjecture holds for imaginary quadratic fields, by doing the same augument as the proof of Theorem 13.31(\cite{Washington}[p.~294]), we find that
\begin{equation*}
X_\infty:=\Gal(M_\infty/K_\infty^-)\sim\Lambda(\Gamma)\oplus(\mbox{torsion parts}).
\end{equation*}
\begin{lem}
There is a surjection
\begin{equation*}
X_\infty\twoheadrightarrow\Lambda(\Gamma)/(T-p)\Lambda(\Gamma)(\simeq\mathbb{Z}_p)
\end{equation*}
\end{lem}
\begin{proof}
We have the exact sequence of $\Lambda(\Gamma)$-modules
\begin{equation*}
0\rightarrow A\rightarrow X_\infty\rightarrow\Lambda(\Gamma)\oplus(\mbox{torsion parts})\rightarrow B\rightarrow0
\end{equation*}
with $A$, $B$ finite. Hence we have the exact sequence
\begin{equation*}
X_\infty/\tor_{\Lambda(\Gamma)} X_\infty\rightarrow\Lambda(\Gamma)\rightarrow B'\rightarrow0
\end{equation*}
with $B'$ finite. By taking $/(T-p)$, we have
\begin{equation*}
(X_\infty/\tor_{\Lambda(\Gamma)} X_\infty)/(T-p)\rightarrow\Lambda(\Gamma)/(T-p)\buildrel{\varphi}\over\rightarrow B'/(T-p)\rightarrow0.
\end{equation*}
Since $B'/(T-p)$ is finite and $\Lambda(\Gamma)/(T-p)\simeq\mathbb{Z}_p$, we must have $\ker\varphi\simeq\mathbb{Z}_p$. Therefore, we obtain the map
\begin{equation*}
X_\infty\twoheadrightarrow X_\infty/(\tor_{\Lambda(\Gamma)} X_\infty+(T-p)X_\infty)\twoheadrightarrow\mathbb{Z}_p.
\end{equation*}
\end{proof}
Thus $M_\infty/K_\infty^-$ has a $\mathbb{Z}_p$-subextension $L/K_\infty^-$. Since any $\Lambda(\Gamma)$-submodule of $X_\infty$ is normal in $\Gal(M_\infty/K)$, one has $\Gal(M_\infty/L)\triangleleft\Gal(M_\infty/K)$. This means that $L/K$ is Galois, and so one finds that $\Gal(L/K)\simeq\mathbb{Z}_p\rtimes\mathbb{Z}_p$.
\begin{lem}
$\Gal(L/K)$ is not abelian.
\end{lem}
\begin{proof}
By Lemma \ref{0110}, we have ``$\Gamma^-$ acts trivially on $\Gal(L/K_\infty^-)\Leftrightarrow\Gal(L/K)$ is abelian." However, for the fixed topological generater $\gamma$ of $\Gamma^-$ and for any nonzero $x\in X$, we have $\gamma\cdot x=(1+T)x=(1+p)x$ in $\Gal(L/K_\infty^-)$. If $x=(1+p)x$, then one has $px=0$, which contradicts $\Gal(L/K_\infty^-)\simeq\mathbb{Z}_p$. Thus we have $\gamma\cdot x\neq x$, and so $\Gal(L/K)$ is not abelian.
\end{proof}

Therefore, we have constructed a non-abelian $\mathbb{Z}_p\rtimes\mathbb{Z}_p$-extension $L/K$. By using this extension, we obtain the following result, which is a partial analogue of Iwasawa's result(\cite{Iwasawa2}[p.~6, Theorem 1.]).

\begin{theorem}\label{0118}
Let $p=3$ and $K:=\mathbb{Q}(\zeta_3)$. For any $N\geq1$, there exists a cyclic extension $K'/K$ with degree $3$ and $L'/K'$ with $\Gal(L'/K')\simeq\mathbb{Z}_p\rtimes\mathbb{Z}_p$ such that $\mu_G(L'/K')\geq N$.
\end{theorem}

\begin{proof}

Since $K=\mathbb{Q}(\sqrt{-3})$, we can use the same argument and the same notations in this section. By a paper of Iwasawa (\cite{Iwasawa2}[p.~2-6]), there exists $d\geq 2$ such that the $\mu$-invariant of $K_\infty^{-}K'/K'$ with $K':=\mathbb{Q}(\zeta_3, \sqrt[3]{d})$ is greater than or equal to $N$. Let $q$ be a prime factor of $d$ other than $3$. Then, by elementary number theory, we find that $q$ is totally ramified in $K'/K$. But $q$ is unramified in $L/K$. Hence we have $L\cap K'=K$, and so $\Gal(LK'/K')\simeq\Gal(L/K)\simeq\mathbb{Z}_p\rtimes\mathbb{Z}_p$. Since the $\mu$-invariant of $K_\infty^{-}K'/K'$ is greater than or equal to $N$, by \cite{Lei}[Proposition 5.1], we obtain that the $\mu_G$-invariant of $LK'/K'$ is also greater than or equal to $N$.
\end{proof}

\section{Determinants related to $\mathbb{Z}_p\rtimes\mathbb{Z}_p$}

In this section, we shall calculate the determinants of some matrices that are related to the groups $\mathbb{Z}_p\rtimes\mathbb{Z}_p$. Let $p$ be a prime number and $n,d\geq 1$ integers. Let $u$ be an integer such that $(p,u)=1$. Let $A(p,n,d,u)=(a_{ij})\in M_{p^n}(\mathbb{Z}_p)$ defined by
\begin{equation*}
a_{ij}=
\begin{cases}
1&j\equiv(i-1)(1+pu)^k+1\mod p^n(0\leq k\leq d-1)\\
0&\text{otherwise}.
\end{cases}
\end{equation*}
\begin{example}
Let $p=5$, $n=2$, $d=3$, and $u=1$. Then $A$ is
\begin{equation*}
\left(\begin{array}{rrrrrrrrrrrrrrrrrrrrrrrrr}
1 & 0 & 0 & 0 & 0 & 0 & 0 & 0 & 0 & 0 & 0 & 0 & 0 & 0 & 0 & 0 & 0 & 0 & 0 & 0 & 0 & 0 & 0 & 0 & 0 \\
0 & 1 & 0 & 0 & 0 & 0 & 1 & 0 & 0 & 0 & 0 & 1 & 0 & 0 & 0 & 0 & 0 & 0 & 0 & 0 & 0 & 0 & 0 & 0 & 0 \\
0 & 0 & 1 & 0 & 0 & 0 & 0 & 0 & 0 & 0 & 0 & 0 & 1 & 0 & 0 & 0 & 0 & 0 & 0 & 0 & 0 & 0 & 1 & 0 & 0 \\
0 & 0 & 0 & 1 & 0 & 0 & 0 & 0 & 1 & 0 & 0 & 0 & 0 & 0 & 0 & 0 & 0 & 0 & 1 & 0 & 0 & 0 & 0 & 0 & 0 \\
0 & 0 & 0 & 0 & 1 & 0 & 0 & 0 & 0 & 0 & 0 & 0 & 0 & 0 & 0 & 0 & 0 & 0 & 0 & 1 & 0 & 0 & 0 & 0 & 1 \\
0 & 0 & 0 & 0 & 0 & 1 & 0 & 0 & 0 & 0 & 0 & 0 & 0 & 0 & 0 & 0 & 0 & 0 & 0 & 0 & 0 & 0 & 0 & 0 & 0 \\
0 & 0 & 0 & 0 & 0 & 0 & 1 & 0 & 0 & 0 & 0 & 1 & 0 & 0 & 0 & 0 & 1 & 0 & 0 & 0 & 0 & 0 & 0 & 0 & 0 \\
0 & 0 & 1 & 0 & 0 & 0 & 0 & 1 & 0 & 0 & 0 & 0 & 0 & 0 & 0 & 0 & 0 & 1 & 0 & 0 & 0 & 0 & 0 & 0 & 0 \\
0 & 0 & 0 & 0 & 0 & 0 & 0 & 0 & 1 & 0 & 0 & 0 & 0 & 1 & 0 & 0 & 0 & 0 & 0 & 0 & 0 & 0 & 0 & 1 & 0 \\
0 & 0 & 0 & 0 & 1 & 0 & 0 & 0 & 0 & 1 & 0 & 0 & 0 & 0 & 0 & 0 & 0 & 0 & 0 & 0 & 0 & 0 & 0 & 0 & 1 \\
0 & 0 & 0 & 0 & 0 & 0 & 0 & 0 & 0 & 0 & 1 & 0 & 0 & 0 & 0 & 0 & 0 & 0 & 0 & 0 & 0 & 0 & 0 & 0 & 0 \\
0 & 0 & 0 & 0 & 0 & 0 & 0 & 0 & 0 & 0 & 0 & 1 & 0 & 0 & 0 & 0 & 1 & 0 & 0 & 0 & 0 & 1 & 0 & 0 & 0 \\
0 & 0 & 0 & 0 & 0 & 0 & 0 & 1 & 0 & 0 & 0 & 0 & 1 & 0 & 0 & 0 & 0 & 0 & 0 & 0 & 0 & 0 & 1 & 0 & 0 \\
0 & 0 & 0 & 1 & 0 & 0 & 0 & 0 & 0 & 0 & 0 & 0 & 0 & 1 & 0 & 0 & 0 & 0 & 1 & 0 & 0 & 0 & 0 & 0 & 0 \\
0 & 0 & 0 & 0 & 1 & 0 & 0 & 0 & 0 & 1 & 0 & 0 & 0 & 0 & 1 & 0 & 0 & 0 & 0 & 0 & 0 & 0 & 0 & 0 & 0 \\
0 & 0 & 0 & 0 & 0 & 0 & 0 & 0 & 0 & 0 & 0 & 0 & 0 & 0 & 0 & 1 & 0 & 0 & 0 & 0 & 0 & 0 & 0 & 0 & 0 \\
0 & 1 & 0 & 0 & 0 & 0 & 0 & 0 & 0 & 0 & 0 & 0 & 0 & 0 & 0 & 0 & 1 & 0 & 0 & 0 & 0 & 1 & 0 & 0 & 0 \\
0 & 0 & 1 & 0 & 0 & 0 & 0 & 0 & 0 & 0 & 0 & 0 & 1 & 0 & 0 & 0 & 0 & 1 & 0 & 0 & 0 & 0 & 0 & 0 & 0 \\
0 & 0 & 0 & 0 & 0 & 0 & 0 & 0 & 1 & 0 & 0 & 0 & 0 & 0 & 0 & 0 & 0 & 0 & 1 & 0 & 0 & 0 & 0 & 1 & 0 \\
0 & 0 & 0 & 0 & 0 & 0 & 0 & 0 & 0 & 1 & 0 & 0 & 0 & 0 & 1 & 0 & 0 & 0 & 0 & 1 & 0 & 0 & 0 & 0 & 0 \\
0 & 0 & 0 & 0 & 0 & 0 & 0 & 0 & 0 & 0 & 0 & 0 & 0 & 0 & 0 & 0 & 0 & 0 & 0 & 0 & 1 & 0 & 0 & 0 & 0 \\
0 & 1 & 0 & 0 & 0 & 0 & 1 & 0 & 0 & 0 & 0 & 0 & 0 & 0 & 0 & 0 & 0 & 0 & 0 & 0 & 0 & 1 & 0 & 0 & 0 \\
0 & 0 & 0 & 0 & 0 & 0 & 0 & 1 & 0 & 0 & 0 & 0 & 0 & 0 & 0 & 0 & 0 & 1 & 0 & 0 & 0 & 0 & 1 & 0 & 0 \\
0 & 0 & 0 & 1 & 0 & 0 & 0 & 0 & 0 & 0 & 0 & 0 & 0 & 1 & 0 & 0 & 0 & 0 & 0 & 0 & 0 & 0 & 0 & 1 & 0 \\
0 & 0 & 0 & 0 & 0 & 0 & 0 & 0 & 0 & 0 & 0 & 0 & 0 & 0 & 1 & 0 & 0 & 0 & 0 & 1 & 0 & 0 & 0 & 0 & 1
\end{array}\right)
\end{equation*}
\end{example}
Let $s,t\in\mathbb{Z}/p^n\mathbb{Z}$. Define a relation $\approx$ in $\mathbb{Z}/p^n\mathbb{Z}$ by
\begin{equation*}
s\approx t\iff\text{There exists } 1\leq r\leq p^n \text{ such that } a_{rs}=1\text{ and }a_{rt}=1.
\end{equation*}
Let $\sim$ be the equivalence relation in $\mathbb{Z}/p^n\mathbb{Z}$ generated by $\approx$.
\begin{lem}\label{0311}
$\sim$ induces the direct sum
\begin{equation*}
\coprod_{\substack{1\leq i\leq p-1\\1\leq m\leq n}}\{ip^{m-1}+1+p^ml|l\in\mathbb{Z}/p^n\mathbb{Z}\}=\mathbb{Z}/p^n\mathbb{Z}.
\end{equation*}
\end{lem}
\begin{proof}
Fix $1\leq i\leq p-1$ and $1\leq m\leq n$. Then we have
\begin{equation*}
(ip^{m-1}+1-1)(1+pu)+1=ip^{m-1}(1+pu).
\end{equation*}
Hence, by the definition of $A$, we have $ip^{m-1}+1\sim ip^{m-1}(1+pu)$. By iterating this, we have $ip^{m-1}+1\sim i p^{m-1}(1+pu)^k$ for all $k\geq0$. Since $1+pu$ generates $1+p(\mathbb{Z}/p^n\mathbb{Z})$, we have $ip^{m-1}(1+pl)=ip^{m-1}+p^mil$ for all $l\in\mathbb{Z}/p^n\mathbb{Z}$. Since $(i,p)=1$, we have $ip^{m-1}+1\sim ip^{m-1}+p^ml$ for all $l\in\mathbb{Z}/p^n\mathbb{Z}$. This shows that we have $ip^{m-1}+1+p^ml_1\sim ip^{m-1}+1+p^ml_2$ for all $l_1,l_2\in\mathbb{Z}/p^n\mathbb{Z}$. On the other hand, if $i_1 p^{m_1-1}+1+p^{m_1}l_1\sim i_2p^{m_2-1}+1+p^{m_2}l_2$ with $ m_1\leq m_2$, then we have $i_1-i_2p^{m_2-m_1}+p(l_1+p^{m_2-m_1-1}l_2)\in p^{n-m_1+1}\mathbb{Z}$. Since $p\nmid i_1$, we must have $m_1=m_2$. This implies $p\mid(i_1-i_2)$, and we must have $i_1-i_2=0$ since $1\leq i_1, i_2\leq p-1$.
\end{proof}
\begin{lem}\label{0312}
Let $\zeta:=\zeta_{p^{n-1}}$ denote a primitive $p^{n-1}$-th root of unity. If $d<p$, then we have
\begin{equation*}
\prod_{k=0}^{p^{n-1}-1}(1+\zeta^k+\cdots+\zeta^{k(d-1)})=d
\end{equation*}
\end{lem}
\begin{proof}
Define a relation $\sim$ in $\mathbb{Z}/p^{n-1}\mathbb{Z}$ by
\begin{equation*}
k_1\sim k_2\iff\text{There exists }m\geq0\text{ such that }k_1d^m\equiv k_2\mod p^{n-1}.
\end{equation*}
Then one can easily check that this is an equivalence relation. Put $l:=|<d>|$ in $\mathbb{Z}/p^{n-1}\mathbb{Z}$. For $1\leq k\leq p^{n-1}-1$, let
\begin{equation*}
C_k:=\prod_{i=0}^{l-1}(1+\zeta^{d^ik}+\zeta^{2d^ik}+\cdots+\zeta^{(d-1)d^ik}).
\end{equation*}
Then we have
\begin{equation*}
(\zeta^k-1)C_k=\zeta^{d^lk}-1=\zeta^k-1.
\end{equation*}
Since $\zeta^k\neq 1$, we have $C_k=1$. Since $\sim$ is an equivalence relation in $\mathbb{Z}/p^{n-1}\mathbb{Z}$, $\prod_{k=0}^{p^{n-1}-1}(1+\zeta^k+\cdots+\zeta^{(d-1)k})$ is a product of $C_k$. Thus we obtain the assertion.
\end{proof}
\begin{theorem}\label{03122}
We have
\begin{equation*}
|A|=
\begin{cases}
1&n=1\\
d^{(p-1)(n-1)}&n\neq1, d<p\\
0&n\neq1, d\geq p.
\end{cases}
\end{equation*}
\end{theorem}
\begin{proof}
Lemma \ref{0311} allows us to consider that $|A|$ is a product of determinants of small matrices. By doing an appropriate base change, we obtain
\begin{equation*}
|A|=
\begin{vmatrix}
 A_{n-1}&&&&&&&&\\
& A_{n-1}&&&&&&&\\
&& \ddots&&&&&&\\
&&& A_{n-1}&&&&&\\
&&&& \ddots&&&&\\
&&&&& A_0&&&\\
&&&&&& A_0&&\\
&&&&&&& \ddots&\\
&&&&&&&& A_0
\end{vmatrix}
\end{equation*}
with
\begin{equation*}
A_{n-m}=
\begin{pmatrix}
1&1&\cdots&1&&&\\
&1&1&\cdots&1&&\\
&&\ddots&\ddots&\ddots&\ddots&\\
&&&1&1&\cdots&1\\
1&&&&1&1&\cdots\\
\vdots&\ddots&&&&\ddots&\vdots\\
1&\cdots&1&&&&1
\end{pmatrix}.
\end{equation*}
By a property of circulant matrices, we have
\begin{equation*}
|A_{n-m}|=\prod_{k=0}^{p^{n-m}-1}(1+\zeta^k+\cdots+\zeta^{k(d-1)}).
\end{equation*}
If $d<p$, then, by Lemma \ref{0312}, we have
\begin{equation*}
|A_{n-m}|=
\begin{cases}
1&n=m\\
d&n\neq m.
\end{cases}
\end{equation*}
If $d\geq p$, then we have
\begin{equation*}
|A_1|=
\begin{vmatrix}
1&\cdots&1\\
\vdots&\ddots&\vdots\\
1&\cdots&1
\end{vmatrix}
=0.
\end{equation*}
Therefore, we obtain
\begin{equation*}
|A|=
\begin{cases}
1&n=1\\
d^{(p-1)(n-1)}&n\neq1, d<p\\
0&n\neq1, d\geq p.
\end{cases}
\end{equation*}
\end{proof}
\begin{example}
Let $p=3$ and equip the relation of $S$ and $T$ in $\Lambda(G)$ by
\begin{equation*}
(1+S)(1+T)=(1+T)(1+S)^4.
\end{equation*}
Consider the left ideal $I:=(\Lambda(G)/\Lambda(G)\omega_2)(T+3)$ of the ring $\Lambda(G)/\Lambda(G)\omega_2$. Then any element $x$ of $I$ can be written in the form
\begin{eqnarray*}
x&=&a_8(1+S)^8(1+T)+\cdots +a_1(1+S)(1+T)+a_0(1+T)\\
&&+a_8(1+S)^8\cdot 2+\cdots+a_1(1+S)\cdot 2+a_0\cdot 2,
\end{eqnarray*}
where $a_i\in\Lambda(\Gamma)$. One has
\begin{eqnarray*}
a_8(1+S)^8(1+T)&=&a_8(1+T)(1+S)^{32}=a_8(1+T)(1+S)^5.                                                                                                                                                                                                                                                                                                                                                                                                                                                                                                                                                                                                                                                                                                                                                                                                                                                                                                                                                                                                                                                                                                                                                                                                                                                                                                                                                                                                                                                                                       \\ a_7(1+S)^7(1+T)&=&a_8(1+T)(1+S)^{28}=a_8(1+T)(1+S).                                                                                                                                                                                                                                                                                                                                                                                                                                                                                                                                                                                                                                                                                                                                                                                                                                                                                                                                                                                                                                                                                                                                                                                                                                                                                                                                                                                                                                                                                    
\\ a_6(1+S)^6(1+T)&=&a_8(1+T)(1+S)^{24}=a_8(1+T)(1+S)^6.                                                                                                                                                                                                                                                                                                                                                                                                                                                                                                                                                                                                                                                                                                                                                                                                                                                                                                                                                                                                                                                                                                                                                                                                                                                                                                                                                                                                                                                                                   
\\ a_5(1+S)^5(1+T)&=&a_8(1+T)(1+S)^{20}=a_8(1+T)(1+S)^2.                                                                                                                                                                                                                                                                                                                                                                                                                                                                                                                                                                                                                                                                                                                                                                                                                                                                                                                                                                                                                                                                                                                                                                                                                                                                                                                                                                                                                                                                                   
\\ a_4(1+S)^4(1+T)&=&a_8(1+T)(1+S)^{16}=a_8(1+T)(1+S)^7.                                                                                                                                                                                                                                                                                                                                                                                                                                                                                                                                                                                                                                                                                                                                                                                                                                                                                                                                                                                                                                                                                                                                                                                                                                                                                                                                                                                                                                                                                   
\\ a_3(1+S)^3(1+T)&=&a_8(1+T)(1+S)^{12}=a_8(1+T)(1+S)^3.                                                                                                                                                                                                                                                                                                                                                                                                                                                                                                                                                                                                                                                                                                                                                                                                                                                                                                                                                                                                                                                                                                                                                                                                                                                                                                                                                                                                                                                                                   
\\ a_2(1+S)^2(1+T)&=&a_8(1+T)(1+S)^8=a_8(1+T)(1+S)^8.                                                                                                                                                                                                                                                                                                                                                                                                                                                                                                                                                                                                                                                                                                                                                                                                                                                                                                                                                                                                                                                                                                                                                                                                                                                                                                                                                                                                                                                                                
\\ a_1(1+S)(1+T)&=&a_8(1+T)(1+S)^4=a_8(1+T)(1+S)^4.
\end{eqnarray*}
Hence we have
\begin{equation*}
x=(a_2(1+T)+2a_8)(1+S)^8+(a_4(1+T)+2a_7)(1+S)^7+\cdots+a_0(1+T+2).
\end{equation*}
Put
\begin{equation*}
b_8:=a_2+a_8,~b_7:=a_4+a_7, \cdots, b_0:=a_0.
\end{equation*}
Then we have
\begin{equation*}
A(3,2,2,1)
\begin{pmatrix}
a_0\\ \vdots\\ a_8
\end{pmatrix}
=
\begin{pmatrix}
b_0\\ \vdots\\ b_8
\end{pmatrix}
\end{equation*}
Since $|A(3,2,2,1)|=4$, $A$ is invertible in $M_{p^2}(\mathbb{Z}_p)$. Therefore, for any pair $(b_0,\cdots,b_8)\in\Lambda(\Gamma)^9$, one can choose $a_0,\cdots,a_8\in \Lambda(\Gamma)$ such that each sum of $a_i$ corresponding to $1,1+S,\cdots,(1+S)^8$ is equal to $b_0, \cdots, b_8$ respectively.
\end{example}

Sohei Tateno\\
Graduate School of Mathematics, Nagoya University, Furocho, Chikusa-ku, Nagoya, 464-8602, Japan\\
E-mail address: inu.kaimashita@gmail.com

\clearpage
\begin{center}
\begin{LARGE}
Corrigendum to ``On Iwasawa's class number formula for $\Z_p\rtimes\Z_p$-extensions"\vskip\baselineskip
\end{LARGE}
\begin{large}
Sohei Tateno\vskip.75\baselineskip
\today\vskip 1.5\baselineskip
\end{large}
\end{center}
\begin{enumerate}

\item
p.86, l.7. We also need to assume that ${\rm Char}_{\Lambda(H)}A$ and ${\rm Char}_{\Lambda(H)}B$ are coprime to $\omega_m$ for each $m\geq 1$.

\item
p.93, l.5. Since we have not shown that our $\Lambda(G)$-module satisfies the $\mathfrak{M}_H(G)$-property, we are not sure if we can apply [6, Proposition 5.1] to our case or not. Also, we need to compare the $\mu$-invariant of our $\Lambda(\Gamma)$-module to that of $\mathcal{X}_{H_m}$ in the sense of [6, Proposition 5.1]. Hence, the proof of Theorem 4.4 is incomplete, and so what we have shown is ``For any $N\geq 1$, there exists a non-abelian $\Z_3\rtimes\Z_3$-extension over a number field whose $\Z_3$-subextension has the $\mu$-invariant greater than or equal to $N$. Since these extensions do not contain cyclotomic $\Z_3$-extensions as their subextensions, we find that their $\mu_G$-invariants are not necessarily zero."

\end{enumerate}
\end{document}